\documentclass[11pt,leqno,letterpaper]{article}

\usepackage[svgnames]{xcolor}
\usepackage{setspace}
\usepackage{abstract}
\usepackage[normalem]{ulem}
\usepackage{mathtools}
\usepackage[abbrev,lite,nobysame]{amsrefs}
\usepackage{amsthm,amsfonts}
\usepackage{thmtools}
\usepackage{csquotes}
\usepackage{enumerate}
\usepackage{bbm}
\usepackage{dsfont}
\usepackage{marvosym}

\usepackage[normalem]{ulem}
\newcommand{\stkout}[1]{\ifmmode\text{\sout{\ensuremath{#1}}}\else\sout{#1}\fi}

\def\cf{\emph{cf.\/}}\def\ie{\emph{i.e.\/}}

\def\etal{\emph{et al.\/}}

\usepackage[T1]{fontenc}
\usepackage[looser]{newtxtext}
\usepackage[bigdelims,varvw]{newtxmath}
\usepackage[protrusion=true, tracking=true, kerning=true,spacing=true]{microtype}
\pdfmapfile{=ntx.map}

\usepackage[colorlinks=true,linkcolor=NavyBlue,
citecolor=NavyBlue,urlcolor=NavyBlue,bookmarksdepth=3,hypertexnames=false]{hyperref}
\usepackage{cleveref}
\crefname{equation}{}{}

\numberwithin{equation}{section}

\newcommand{\boxnote}[1]{\makebox[0mm][l]{$#1$}}

\theoremstyle{plain}
\newtheorem{theorem}{Theorem}[section]
\newtheorem{lemma}[theorem]{Lemma}

\newtheorem{proposition}[theorem]{Proposition}
\newtheorem{example}[theorem]{Example}
\newtheorem{assumption}[theorem]{Assumption}
\newtheorem{remark}[theorem]{Remark}
\theoremstyle{definition}
\newtheorem{definition}[theorem]{Definition}

\renewcommand{\leq}{\leqslant}

\renewcommand{\geq}{\geqslant}
\renewcommand{\mathcal}{\mathscr}
\renewcommand{\tilde}{\widetilde}

\renewcommand{\longrightarrow}{\to}

\newcommand\tfootnote[1]{%
	\begingroup
	\renewcommand\thefootnote{}\footnote{#1}%
	\addtocounter{footnote}{-1}%
	\endgroup
}

\begin{document}\linespread{1.05}\selectfont
	\date{}

	\author{Manuel~Rissel \protect\footnote{Institute of Mathematical Sciences, ShanghaiTech University, Shanghai 201210, China, e-mail: \href{mailto:mrissel@shanghaitech.edu.cn}{mrissel@shanghaitech.edu.cn}} \and Marius~Tucsnak  \protect\footnote{Institut de Math\'ematiques de Bordeaux, UMR 5251, Universit\'{e} de Bordeaux/Bordeaux INP/CNRS, 351 Cours de la Lib\'eration - F 33 405 TALENCE, France, and Institut Universitaire de France (IUF), e-mail: \href{mailto:marius.tucsnak@u-bordeaux.fr}{Marius.Tucsnak@u-bordeaux.fr}}}

	\title{Approximate Tracking Controllability of Systems with Quadratic Nonlinearities}
	
	\maketitle
	
	\tfootnote{{{\bf MSC2020}: 93B05 (primary); 93C10 (secondary).} {\bf Keywords}:  tracking controllability, relaxation norm, quadratic nonlinearities, dynamic motion planning}

\begin{abstract}
	Given a finite-dimensional time continuous control system and $\varepsilon>0$, we address the question of the existence of controls that maintain the corresponding state trajectories in the $\varepsilon$-neighborhood of any prescribed path in the state space. We investigate this property, called approximate tracking controllability, for linear and quadratic time invariant systems. Concerning linear systems, our answers are negative: by developing a systematic approach, we demonstrate that approximate tracking controllability of the full state is impossible even in a certain weak sense, except for the trivial situation where the control space is isomorphic to the state space.  Motivated by these negative findings for linear systems, we focus on nonlinear dynamics. In particular, we prove weak approximate tracking controllability on any time horizon for a general class of systems with arbitrary linear part and quadratic nonlinear terms.  The considered weak notion of approximate tracking controllability involves the relaxation metric. We underline the relevance of this weak setting by developing applications to coupled systems (including motion planning problems) and by remarking obstructions that would arise for natural stronger norms. The exposed framework yields global results even if the uncontrolled dynamics might exhibit singularities in finite time. 
\end{abstract}

\newpage
	
\setcounter{tocdepth}{3}	
\tableofcontents

\section{Introduction}\label{section:introduction}

The main objective of controllability theory of time-continuous dynamical systems consists in studying the existence of input signals driving the state trajectory from an arbitrary initial state to a prescribed final state.  This theory is developed in an overwhelming
number of papers and monographs, for both finite- and infinite-dimensional systems, see for instance, Coron \cite{coron2007control} and references therein. Much less is done on a property called {\em tracking controllability}, where the aim is controlling the full state of the system or an output function at each instant $t$ in a fixed time interval $[0,\tau]$. The exact or approximate output tracking have been developed in series of papers, see, for instance, Garc\'ia-Planas and Dom\'inguez-Garc\'ia \cite{garcia2013alternative},
or Zamorano and Zuazua \cite{zamorano2024tracking} for the linear case. To our best knowledge there are no systhematic studies
of the tracking controllability of the full state. This is probably due to the fact that this is a very strong property, which seems out of reach in many classical situations. In fact, as we show in \Cref{section:l} below, for linear systems this property  holds (even in a weak sense)  only in the trivial case where the control operator is onto. Motivated by this negative result, we focus on a class of nonlinear systems for which approximate tracking controllability (in the same weak sense) holds. 

To be more precise, let $U$ (the control space), $X$ (the state space)  be finite-dimensional inner product spaces. We consider control systems described by 
\begin{equation}\label{equation:main_system}
	\begin{gathered}
		\dot{x}(t) + Ax(t) + f(x(t)) = Bu(t)  \boxnote{\qquad \qquad \, (t\in [0,\tau]),}
	\end{gathered}
\end{equation}
with the initial condition
\[
	x(0) = x_0 \in X.
\]
where $\tau>0$ is the control time, $x\colon [0,\tau]\to X$ is the state trajectory,  $u \colon [0,\tau] \longrightarrow U$ is the control function, $A \in \mathcal{L}(X)$ and $B \in \mathcal{L}(U, X)$ are linear maps from $X$ to $X$ and from $U$ to $X$, respectively, and the nonlinearity $f\colon X\longrightarrow X$ is supposed locally Lipschitz. Questions of (tracking-) controllability for such time-invariant systems arise in various applications, including complex networks \cite{Liuetal2011}, robotics \cite{LiuShenyu_etal2019}, and machine learning research \cite{Ruiz-Balet2023}; see also the respective bibliographies of these references.

The goal of the present article is to provide elementary and checkable conditions on $f$, $U$, and $B$ that ensure for any $\varepsilon > 0$ and any time $\tau > 0$ the existence of controls $u\colon[0,\tau]\to U$ which maintain the corresponding solutions to \eqref{equation:main_system} on the whole time interval $[0,\tau]$ in the $\varepsilon$-neighborhood (with respect to a suitable norm) of a prescribed function $\psi\colon [0,\tau]\longrightarrow X$. In other words, we investigate under what assumptions appropriately forced solutions to \eqref{equation:main_system} can approximately track, with respect to a suitable norm, any curve in the state space. The nature of the results obtained here will depend crucially on the considered norms and the choices of $f$, $U$, and $B$.

More precisely, the main controllability property of \eqref{equation:main_system} investigated in this article is defined below.

\begin{definition}\label{def_gen_norm}
Let $\tau>0$ and let $N_{\tau}$ be a norm on $L^2([0,\tau];X)$. The system \eqref{equation:main_system} is called approximately tracking controllable on $[0,\tau]$ with respect to $N_\tau$, if for every $\psi\in W^{1,2}((0,\tau);X)$ and $\varepsilon>0$, there exists a control function $u\in L^2([0,\tau]; U)$ such that the solution~$x$ of \eqref{equation:main_system} with $x(0)=\psi(0)$ satisfies 
\begin{equation}\label{approximate_norm_gen}
N_\tau(x-\psi)\leqslant \varepsilon.
\end{equation}  
\end{definition}

\

\begin{remark}
    The property \eqref{approximate_norm_gen} of \Cref{def_gen_norm} is independent of the initial state of~$x$. Indeed, given $\tau>0$ and $\psi\in W^{1,2}((0,\tau);X)$, assume that $x_1 \in X$ is chosen such that $x_1 \neq \psi(0)$. Then, for any $\varepsilon > 0$ fix $\widetilde{\psi} \in W^{1,2}((0,\tau);X)$ with $\widetilde{\psi}(0) = x_1$ and $N_{\tau}(\widetilde{\psi}-\psi) \leq \varepsilon/2$. Now, we can infer from the assumption that \Cref{def_gen_norm} holds the existence of $\widetilde{u} \in L^2([0,\tau]; U)$ such that the corresponding solution $\widetilde{x}$ to \eqref{equation:main_system} with $\widetilde{x}(0) = x_1$ satisfies $N_{\tau}(\widetilde{x} - \widetilde{\psi}) \leq \varepsilon/2$. This implies $N_{\tau}(\widetilde{x} - \psi) \leq \varepsilon$.
\end{remark}

\begin{remark}
    In \Cref{def_gen_norm}, it is equivalent to require $u\in C^{\infty}([0,\tau]; U)$. This follows by combining a density argument with a standard result on the continuous dependence of solutions to \eqref{equation:main_system} on the data (see \Cref{proposition:perturbativestability} below).
\end{remark}

As far as we know, the property defined above in \Cref{def_gen_norm}, where $X$ is finite-dimensional,x has not been systematically studied in the literature, namely in the nonlinear case. For linear systems; i.e. $f=0$ in \eqref{equation:main_system},
a more general concept appears in Definition 4.1 from  \cite{zamorano2024tracking}. There, $N_\tau$ in the above mentioned definition is chosen to be the $L^2([0,\tau];X)$ norm and the focus is on the characterization by duality of the defined concept; see \Cref{section:l} below for more details. To our best knowledge, for nonlinear finite-dimensional systems, the approximate tracking controllability (with $N_\tau$ being the usual norm in $L^2([0,\tau];X)$) has been considered only for affine control systems without drift, see, for instance, Liu and Sussmann \cite{sussmann1991limits},
mainly as a tool in obtaining controllability results \cite{haynes1970nonlinear,sussmann1991limits,liu1997approximation,Agrachev2024}. Let us also mention the systematic study of small-time local controllability for scalar-input systems due to Beauchard and Marbach in \cite{BeauchardMarbach2018}, which particularly highlights the role played by the choice of the norms to measure relevant quantities such as the \enquote{size} of the controlled state trajectory. In an infinite-dimensional context, some approximate tracking controllability results (with respect to a weaker norm; see \eqref{def_norm_relax} below) have been given in Koike \etal~\cite{KoikeNersesyanRisselTucsnak2025} and Nersesyan \cite{Ners-2015} for PDEs describing fluid flows. The method used there, and in parts of the present article, arises in a broader context from the Agrachev-Sarychev approach and its refinements; see Agrachev and Sarychev \cite{AS-05,AS-2006}, Shirikyan \cite{shirikyan-cmp2006}, Nersisyan \cite{nersisyan-2010}, or for a general class of parabolic PDEs the work \cite{Nersesyan2021} by Nersesyan. In \cite{BoscainCaponigroSigalotti2014}, tracking in modulus (also called tracking up to phases) has been established for a multi-input Schr\"odinger equation by Boscain, Caponigro, and Sigalotti.

 In this work we focus on two choices of the norm $N_\tau$. The first one is the standard norm in $L^2([0,\tau];U)$. The second norm we consider, called {\em relaxation norm}
(\cf~Gamkrelidze~\cite[p.55]{G-78} and \cite{AS-2006,Ners-2015, KoikeNersesyanRisselTucsnak2025}), is denoted by $|||\cdot|||_\tau$. For $\tau>0$ and $(E, \|\cdot\|_E)$ a Banach space, 
this norm is for all $v\in L^1([0,\tau];E)$ defined by
\begin{equation}\label{def_norm_relax}
\begin{aligned}
    ||| v |||_\tau \coloneqq \sup_{t\in [0,\tau]}  \left\|\int_0^t v(s)\, {\rm d}s\right\|_E.
\end{aligned}
\end{equation}

\begin{definition}\label{def_gen_norm_bis}
Let $\tau>0$. The system \eqref{equation:main_system} is called approximately tracking controllable on $[0,\tau]$
if it satisfies \Cref{def_gen_norm} with $N_\tau$ being the standard norm in $L^2([0,\tau];X)$. Moreover,
\eqref{equation:main_system} is called weakly approximately tracking controllable on $[0,\tau]$
if it satisfies \Cref{def_gen_norm} with $N_\tau$ being the relaxation norm introduced in \eqref{def_norm_relax}.
\end{definition}

 We first discuss the approximate tracking controllability for linear systems; see \Cref{section:l},
 where we prove that weak approximate tracking controllability is achievable only in the ``trivial'' case when $B$ is onto.  
 More precisely, we have:

 \begin{proposition}\label{numai_una_strong}
Assume that that $f=0$ in \eqref{equation:main_system}, so that the considered system is linear. Then, \eqref{equation:main_system} is weakly approximately tracking controllable if and only if the control operator $B$ is onto. 
\end{proposition}

The negative result above motivates the investigation of approximate tracking controllability for nonlinear systems. In this work we consider the case where $f$ in \eqref{equation:main_system} is quadratic.
More precisely, we suppose that we have:

\begin{assumption}\label{assumption:0}
The function $f$ is quadratic in the sense that $f(x)=\Gamma(x,x)$ for a bilinear function $\Gamma \colon X\times X\to X$. Without loss of generality, we can assume  that $\Gamma$ is symmetric, \ie, that $\Gamma(x,y)=\Gamma(y,x)$ for every $x,\ y \in X$. In particular, this means that
\begin{equation}\label{equation:qc}
	f(a) + f(b) = \frac{f(a+b)+f(a-b)}{2}
\end{equation}
for all $a,b \in X$. 
\end{assumption}

No assumption will be made that would rule out finite-time blowup of the uncontrolled dynamics. It will thus be part of our choice of the controls to ensure that the considered solutions are defined on the desired time interval.  Quadratic nonlinearities arise in various mathematical models or approximations thereof. In particular, second order expansions can provide more accurate approximations of general nonlinear systems compared to linearizations alone. Moreover, systems with polynomial nonlinearities can be transformed to quadratic ones by introducing new variables~\cite{Bychkovetal2024}; however, this approach might not be feasible in the presence of controls, as the controllability of the \enquote{quadratified} system might be less clear.
Another, mostly illustrative, reason for choosing quadratic nonlinearities is that they can be seen as toy models for the nonlinearities occurring in the Euler or Navier-Stokes equations in fluid dynamics, or in other infinite-dimensional systems.  It is also possible to apply the techniques presented here to Galerkin approximations of fluid PDEs or other models, but we omit a detailed study of this situation; we refer to \cite{AS-2006} and the work \cite{LionsZuazua1998} by Lions and Zuazua regarding some existing controllability results for Galerkin approximations.  It should be emphasized that the techniques presented here are tailored to quadratic nonlinearities, and extensions to non-quadratic settings are expected to be nontrivial. However, in view of the proof of \Cref{proposition:perturbativestability}, we could allow nonlinear perturbations of non-quadratic type under the strong assumption that they are small in dependence on all data and the choices of $\tau$ and $\varepsilon$ in \Cref{def_gen_norm}.

In order to state a positive result in the presence of quadratic nonlinearities, we need  extra assumptions on the nonlinearity and on the control operator. The precise formulation of these assumptions in a general form requires some preparation, so we postpone it to \Cref{section:approxtrackingnonlinear}. For this introductory section, we limit ourselves to the following stronger but easy to state (and in many cases easy to check) hypothesis: 

\begin{assumption}\label{assumption:1}
	For each $\gamma \in X$ there exist $u, \xi\in U$ such that $\gamma = Bu - f(B\xi)$.  
\end{assumption}

We state below a simplified version of our main result given later in the form of \Cref{theorem:approximatetracking}, referring to \Cref{subsection:pa1} for its proof.

\begin{proposition}\label{theorem:approximatetracking_simp}
	Let \Cref{assumption:0} and \Cref{assumption:1} be satisfied. Then for every $\tau>0$ the system \eqref{equation:main_system} is weakly approximately tracking controllable on $[0,\tau]$.
More precisely, given any $\psi \in W^{1,2}((0,\tau); X)$ and $\varepsilon > 0$, there exists a control $u \in C^{\infty}([0,\tau];U)$ such that the corresponding solution $x$ to~\eqref{equation:main_system} with $x_0 = \psi(0)$ satisfies
	\[
		|x(\tau) - \psi(\tau)| + ||| x-\psi |||_\tau < \varepsilon.
	\]
\end{proposition}

Note that \Cref{theorem:approximatetracking_simp} also provides the global approximate controllability of \eqref{equation:main_system}.

A natural question is whether stronger notions of approximate tracking controllability can be achieved under the assumptions in \Cref{theorem:approximatetracking_simp}. 
The example below shows that the answer is generally negative.

\begin{example}\label{example:00}
	Adhering to the notations from \Cref{equation:main_system}, we consider the following controlled nonlinear system with quadratic drift:
	\begin{equation}\label{equation:net}
		\dot{x}_1 +  x_1x_2 = u_1, \quad \dot{x}_2 + x_3^2 - x_1^2 = 0, \quad \dot{x}_3 - x_3x_2 = u_2.
	\end{equation} 
    The above system clearly meets the hypothesis of \Cref{theorem:approximatetracking_simp} with $X=\mathbb{R}^3$,
    $U=\mathbb{R}^2$, $A=0$, $Bu= [u_1, 0, u_2]^{\top}$ for $u = [u_1, u_2]^{\top} \in U$, and $f(x_1,x_2,x_3)= [x_1x_2,  x_3^2 - x_1^2,- x_3x_2]^{\top}$ for $x = [x_1, x_2, x_3]^{\top}\in X$.
    Indeed, \eqref{equation:net} obviously satisfies \Cref{assumption:0}, whereas the fact that \Cref{assumption:1} holds follows from some simple algebraic manipulations.
    Consequently, the system \eqref{equation:net} is weakly approximately tracking controllable by \Cref{theorem:approximatetracking_simp}.
    
    To prove that it is not approximately tracking controllable we choose, for simplicity, $\tau = 1$ and $x(0)=0$.
    Assume, for a contradiction argument, that for $\psi \coloneq [\psi_1, 0, 0]^{\top}$ with $\psi_1 \in C^{\infty}([0,1]; \mathbb{R}_+) $,
    \(\psi_1(0)=0\) and $\psi_1(t) = 1000$ for all $t \in [1/3, 2/3]$ there would exist controls $(u^n_1,u^n_2)_{n \in \mathbb{N}} \subset L^2((0,1);\mathbb{R}^2)$ such that the associated solutions $(x^n)_{n \in \mathbb{N}} \subset C([0,1]; \mathbb{R}^3)$ to \eqref{equation:net} satisfy
	\[
		\lim\limits_{n \to \infty}\|x^n - \psi\|_{L^2((0,1); \mathbb{R}^3)} = 0. 
	\]
	Then, there would exist a subsequence $(x^{n_k})_{k \in \mathbb{N}} \subset (x^n)_{n \in \mathbb{N}} $ converging pointwise to~$\psi$. Thus, we can fix $k \in \mathbb{N}$ such that $x^{n_k}$, which we for simplicity denote again as $x = [x_1,x_2,x_3]^{\top}$, satisfies $|x(t) - \psi(t)| < 1$ for almost all $t \in (0,1)$.
	Hence, for almost all $t \in [0,1/3]\cup[2/3, 1]$ it would hold $x_1(t) \geq -1$, for almost all $t \in [1/3, 2/3]$ one would have $x_1(t) \geq 999$, and $x_3(t) \in [-1,1]$ would be true for almost all $t \in [0,1]$. Thus, by \eqref{equation:net}, it would follow that $\dot{x}_2 = x_1^2 - x_3^2$ satisfies $\dot{x}_2(t) \geq 998000$ for almost all $t \in [1/3, 2/3]$ and $\dot{x}_2(t) \geq - 1$ for almost all $t \in [0,1/3]\cup[2/3,1]$. However, this would contradict $\|x_2\|_{_{L^2((0,1); \mathbb{R})}} < 1$, hence it would contradict the existence of the index $k$.
\end{example}
	\begin{remark}
		The same argument as in \Cref{example:00} can also be used to obtain negative approximate tracking controllability results with respect to the $L^p$-norm with $p \geq 1$. 
	\end{remark}

\emph{Organization of the remaining part of this article.} \Cref{section:l} demonstrates the lack of weak approximate tracking controllability for linear systems that are of the form \eqref{equation:main_system} with $f \equiv 0$ and $B$ not onto. Our main result, which is \Cref{theorem:approximatetracking}, is stated in \Cref{section:approxtrackingnonlinear} and some examples, including the Lorenz system, are discussed. \Cref{sec_extended} describes basic stability properties of an enlarged system driven by additive and multiplicative controls. An auxiliary trajectory approximation result for the enlarged system is establish in \Cref{new_sec_approx}, and, based on this, the proof of \Cref{theorem:approximatetracking} is completed in \Cref{subsection:pa1}. Applications of \Cref{theorem:approximatetracking} to coupled systems, dynamic control, and motion planning are presented in \Cref{section:dc}.

\section{On the lack of approximate tracking controllability for linear systems}\label{section:l}

Let $U,X$ and $Y$ be finite-dimensional inner product spaces. We consider finite-dimensional linear time-invariant
 systems (LTIs)  with
input space $U$, state space $X$ and output space $Y$ which are described by
\begin{equation}\label{SystemDiffEq_lin}
	\begin{cases}
    	\dot z(t) = A z(t) + B u(t) \boxnote{\qquad\qquad(t\geqslant 0),}\\
        y(t) = C z(t) \boxnote{\qquad\qquad\qquad \quad \, \,(t\geqslant 0),}
	\end{cases}
\end{equation}
where $u(t)\in U$, $u$ is the {\em input function}\index{input
function} (or input signal), $z(t)\in X$ is its {\em
state} at time\index{time} $t$ and $y$ is
the {\em output function} (or output signal).  In the above equations, $A,B,C$ are linear operators such that $A\colon X\to X$,
$B\colon U\to X$ and $C\colon X \to Y$. The definition below is a slight variation of Definition 4.1 from  \cite{zamorano2024tracking}.

\begin{definition}\label{def_linear}
The system \eqref{SystemDiffEq_lin} is approximately $C$-tracking controllable in time $\tau>0$ if for some  $z(0)\in X$, any output target $\tilde y\in L^2([0,\tau]; Y)$, and for each $\varepsilon>0$ there exists a control function $u\in L^2([0,\tau]; U)$ such that  
\begin{equation}\label{approximate}
\|y-\tilde y\|_{L^2([0,\tau];Y)}\leqslant \varepsilon.
\end{equation}  
\end{definition}

\begin{remark}\label{numai_zero}
One can check that the above definition is independent of the choice of the initial state $z(0)\in X$. Indeed, assume that
\Cref{def_linear} holds with $z(0)=z_0$, for some fixed $z_0\in X$. Then, for every $z_1\in X$ and
$\tilde y\in L^2([0,\tau]; Y)$ there exists a control $u_1\in L^2([0,\tau];U)$ such that the solution $\tilde z$ to the problem
\begin{equation*}
\begin{cases}
    \dot{\widetilde z} = A \widetilde{z}+Bu_1,\\
    y_1 = C\widetilde{z},\\
    \widetilde{z}(0)=z_0,
\end{cases}
\end{equation*} 
satisfies
\begin{equation*}
\|y_1-\tilde y-C\exp(tA)(z_0-z_1)\|_{L^2([0,\tau];Y)}\leqslant \varepsilon,
\end{equation*} 
so that
\begin{equation}\label{newoutput}
\|C\exp(tA)z_1+C\Phi_t u_1-\tilde y\|_{L^2([0,\tau];Y)}\leqslant \varepsilon,
\end{equation} 
where 
\begin{equation*}
    \begin{aligned}
        \Phi_t u_1=\int_0^t \exp\left[(t-\sigma)A\right] B u_1(\sigma)\, {\rm d}\sigma \boxnote{ \qquad \qquad (t\in [0,\tau]).}
    \end{aligned}
\end{equation*}
Noting that
\[
	\overline z(t)=\exp(tA)z_1+\Phi_t u_1  \boxnote{\qquad\qquad(t\in [0,\tau]),}
\]
satisfies $\dot {\overline{z}}=A{\overline{z}}+Bu_1$ on $[0, \tau]$ and $\overline z(0)=z_1$, the
estimate \eqref{newoutput} becomes $\|C\overline z-\tilde y\|_{L^2([0,\tau];Y)}\leqslant \varepsilon$. We have thus shown that \eqref{SystemDiffEq_lin} also satisfies \Cref{def_linear} with $z(0)=z_1$.

Consequently, it suffices to take $z(0)=0$, so that the system \eqref{SystemDiffEq_lin} is approximately $C$-tracking controllable  if an only if  for every $\tau>0$ the range of the operator $\mathbb{F}_\tau\in \mathcal{L}\left((L^2([0,\tau];U),L^2([0,\tau];Y)\right)$ defined by 
\begin{equation}\label{equation:ftau}
	\begin{aligned}
		\left(\mathbb F_\tau u\right)(t)=C \int_0^{t} \exp((t-\sigma)A) Bu(\sigma)\, {\rm d}\sigma   \boxnote{\qquad(t\in [0,\tau]),}
	\end{aligned}
\end{equation}
is dense in $L^2([0,\tau];Y)$.
\end{remark}

\begin{remark}\label{rem_sinteaza}
Using \Cref{numai_zero}, it follows that \eqref{SystemDiffEq_lin} satisfies \Cref{def_linear} with $Y=X$ and $C=\mathbb{I}_X$ (the identity in $X$) if and only if this system is approximately tracking controllable
in the sense of \Cref{def_gen_norm}.
\end{remark}

We also recall  the following version of Theorem 4.2 in \cite{zamorano2024tracking} which gives the adjoint of $\mathbb{F}_\tau$.

\begin{proposition}\label{th_4.2}
The adjoint of the operator $\mathbb{F}_\tau$ defined in \eqref{equation:ftau} is given by the operator 
\[
    \Psi_\tau\in \mathcal{L}(L^2([0,\tau];Y),L^2[0,\tau];U))
\]
defined by
\begin{equation}\label{adj_f_en}
	\begin{aligned}
		\left(\Psi_\tau g\right)(\sigma)=B^* \left[\int_\sigma^\tau \exp((t-\sigma)A^*) C^* g(t)\, {\rm d}t\right]
	\end{aligned}
\end{equation}
\[
	\boxnote{\qquad \quad \, (g\in L^2([0,\tau],Y),\sigma\in [0,\tau]).}
\]
Consequently, the system \eqref{SystemDiffEq_lin} is approximately $C$-tracking approximately controllable if and only if ${\rm Ker}\, \Psi_\tau=\{0\}$.
\end{proposition}

\begin{proof}
Let $g\in L^2([0,\tau],Y)$ and $u\in L^2([0,\tau],Y)$.
Then 
\begin{multline*}
\left\langle\mathbb{F}_\tau u,g\right\rangle_{L^2([0,\tau],Y)}=\int_0^\tau \int_0^t \left\langle  u(\sigma), B^*\exp((t-\sigma)A) C^* g(t)\right\rangle_U \, {\rm d}\sigma
\, {\rm d}t\\
= \int_0^\tau  \left\langle  u(\sigma), \int_\sigma^\tau B^*\exp((t-\sigma)A^*) C^* g(t)\, {\rm d}t \right\rangle_U 
\, {\rm d}\sigma,
\end{multline*}
so that $\mathbb{F}_\tau^*$ is indeed the operator $\Psi_\tau$ defined in \eqref{adj_f_en}.

The last assertion in the proposition follows from the fact that $\mathbb{F}_\tau$ is onto  if and only if ${\rm Ker}\, \mathbb{F}_\tau^*=\{0\}$.
\end{proof}

\begin{proposition}\label{numai_una}
For $\tau>0$, the system  \eqref{SystemDiffEq_lin} is  approximately tracking controllable in time $\tau$ if and only if $B$ is onto.
\end{proposition}

\begin{proof}
We first remark that for $Y=X$ and $C=\mathbb{I}_X$ \Cref{th_4.2} yields that $\mathbb{F}_\tau^*$ is given by 
\begin{equation*}\label{adj_f_bis}
\left(\mathbb{F}_\tau^*g\right)(\sigma)=B^* \left[\int_\sigma^\tau \exp((t-\sigma)A^*)  g(t)\, {\rm d}t\right]
\end{equation*} 
\[
	\boxnote{\qquad \quad \, (g\in L^2([0,\tau],X),\sigma\in [0,\tau]).}
\]
Assume that $B$ is not onto. Then there exists $\eta\in X\setminus \{0\}$ such that $B^*\eta=0$. Let
\[
    g(t)= - \alpha(t) A^*\eta-\dot\alpha(t)\eta  \boxnote{\qquad\qquad(t\geqslant 0),}
\]
with $\alpha(t)=(\tau-t)\eta$.  Then, we have $g\neq 0$ and
\begin{equation*}
	\begin{aligned}
		\int_\sigma^\tau \exp((t-\sigma)A^*)  g(t)\, {\rm d}t & =-\int_\sigma^\tau \exp((t-\sigma)A^*)  
		\left[\alpha(t) A^*\eta+\dot\alpha(t)\eta\right]\\
		& =-\int_\sigma^\tau \frac{\rm d}{{\rm d}t} \left[\alpha(t) \exp((t-\sigma)A^*)  \eta\right] \, {\rm d}t=(\tau-\sigma)\eta
	\end{aligned}
\end{equation*}
\[
	\boxnote{\qquad\qquad\qquad\qquad\qquad \, \, \, (\sigma\in [0,\tau]),}
\]
so that $\mathbb{F}_\tau^*g=0$. We have thus shown, by contraposition, that if \eqref{SystemDiffEq_lin} is  approximately tracking controllable
then $B$ is onto.

The proof of the fact that if $B$ is onto then \eqref{SystemDiffEq_lin} is approximately tracking controllable is left for the reader.
\end{proof}

\Cref{def_linear} of approximate $C$-tracking controllability depends on the choice of the norm in the left hand side of \eqref{approximate}.
Therefore, a natural question is whether this property could hold, without assuming that $B$ is onto, when we replace the $L^2([0,\tau];Y)$ norm appearing in \eqref{approximate}
by a weaker norm. A natural candidate in this direction is the {\it relaxation norm}, which has been defined in \eqref{def_norm_relax}.
Using the above norm, we introduce the concept of {\em weakly approximately $C$-tracking controllability} as follows.

\begin{definition}\label{def_linear_weak}
The system \eqref{SystemDiffEq_lin} is weakly approximately $C$-tracking controllable if for every $\tau>0$, $z(0)\in X$, any output target $\tilde y\in L^2([0,\tau]; Y)$, and for every $\varepsilon>0$ there exists a control function $u\in L^2([0,\tau]; U)$ such that  
\begin{equation*}
|||y-\tilde y|||_\tau\leqslant \varepsilon.
\end{equation*} 
\end{definition}

\begin{remark}\label{rem_inca}
Using \Cref{numai_zero}, it follows that the above property holds with $Y=X$ and $C=\mathbb{I}_X$  if and only if the system \eqref{SystemDiffEq_lin} is approximately tracking controllable in time $\tau$
(in the sense of \Cref{def_gen_norm_bis}). 
\end{remark}

\begin{remark}\label{rem_combin}
From Definitions~\ref{def_linear} and~\ref{def_linear_weak}, it follows that the system \eqref{SystemDiffEq_lin} is weakly approximately $C$-tracking controllable in time $\tau$ if an only if  the extended system
\begin{equation}\label{SystemDiffEq_ext}
	\begin{cases}
		\dot{z}_1(t)  = A z_1(t) + B u(t) \boxnote{\qquad (t\geqslant 0),}  \\
		\dot{z}_2(t)  =z_1(t) \boxnote{\qquad \qquad \qquad (t\geqslant 0),} \\
		y(t)  = C z_2(t) \boxnote{\qquad \qquad \quad \,\,\,\,  (t\geqslant 0),} 
	\end{cases}
\end{equation}
with state trajectory $z= [z_1, z_2]^{\top}$, control function $u$, and output signal $y$, is approximately $\tilde C$-tracking controllable
in time $\tau$, with  $\tilde C=\begin{bmatrix} 0,C\end{bmatrix}$.
\end{remark}

We are now in the position to give the main proof in this section.

\begin{proof}[Proof of \Cref{numai_una_strong}]
The fact that if $B$ is onto then \eqref{SystemDiffEq_lin} is  weakly approximately tracking controllable follows from
\Cref{numai_una}.

To prove the converse of the above assertion, we first remark that for $C=\mathbb{I}_X$ the system \eqref{SystemDiffEq_ext} writes
\begin{equation*}\label{SystemDiffEq_extins}
	\begin{cases}
		\dot z(t) = \tilde A z(t) + \tilde B u(t) \boxnote{\qquad (t\geqslant 0),}\\
		y(t) = \tilde C z(t) \boxnote{\qquad \qquad \quad \, \, (t\geqslant 0),}
	\end{cases}
\end{equation*}
where
\begin{equation}\label{defatilde}
    \tilde A=\begin{bmatrix}
        A & 0\\ \mathbb{I}_X & 0
    \end{bmatrix}, \qquad \widetilde{B}^*=\begin{bmatrix}
        B^* & 0
    \end{bmatrix},
    \qquad \widetilde{C}^* = \begin{bmatrix}
        0  \\ \mathbb{I}_X
    \end{bmatrix}.
\end{equation}
By combining \Cref{th_4.2} and \Cref{rem_combin}, it follows that \eqref{SystemDiffEq_lin} is  weakly approximately tracking controllable
in time $\tau$ 
if and only if the operator
\[
\tilde\Psi_\tau\in \mathcal{L}(L^2([0,\tau];Y),L^2[0,\tau];U)),
\]
defined by
\begin{equation*}\label{adj_f_en_bis}
\left(\tilde \Psi_\tau g\right)(\sigma)={\tilde B}^* \left[\int_\sigma^\tau \exp((t-\sigma){\tilde A}^*) {\tilde C}^* g(t) \, {\rm d}t\right]
\end{equation*}
\[
    \boxnote{\qquad \quad \, (g\in L^2([0,\tau],Y),\sigma\in [0,\tau]),} 
\]
satisfies
\begin{equation}\label{ker_e_zer0}
{\rm Ker}\, \tilde\Psi_\tau=\{0\}.    
\end{equation}
Since the definition of $\widetilde{A}$ in \eqref{defatilde} yields
\begin{equation*}
\exp(t{\tilde A}^*)=\begin{bmatrix}
    \exp(tA^*) & \int_0^t \exp((t-s)A^*)\, {\rm d}s\\
    0          &    0
\end{bmatrix}  \boxnote{\qquad(t\in \mathbb{R}),}
\end{equation*}
we can use the definitions of $\widetilde{B}^*$ and $\widetilde{C}^*$ in \eqref{defatilde} to obtain for all $g\in L^2([0,\tau],Y)$ and $\sigma\in [0,\tau]$ that 
\begin{equation}\label{adj_f_en_bis_bis}
    \begin{aligned}
        \left(\tilde \Psi_\tau g\right)(\sigma) &=B^* \left[\int_\sigma^\tau \left(\int_0^{t-\sigma}\exp((t-\sigma-s)A^*)\, {\rm d}s\right)  g(t) \, {\rm d}t\right]\\
& = B^* \left[\int_0^{\tau-\sigma} \int_{s+\sigma}^\tau\exp((t-\sigma-s)A^*)  g(t) \, {\rm d}t\, {\rm d}s\right]\\
& = B^* \left[\int_0^{\tau-\sigma} \int_\xi^\tau\exp((t-\xi)A^*)  g(t) \, {\rm d}t\, {\rm d}\xi\right].
    \end{aligned}
\end{equation}

In what follows, we use a contraposition argument; more precisely, we show that, if~$B$ is not onto, then \eqref{ker_e_zer0} does not hold.
To this end, we remark that, if~$B$ is not onto, then there exists $\eta\in X\setminus \{0\}$ such that $B^*\eta=0$. In this case, we denote $\alpha(t) = (\tau-t)\eta$ and define $g(t)= - \alpha(t) A^*\eta-\dot\alpha(t)\eta$ for all $t\geqslant 0$. 
In particular, it holds $g\neq 0$, and for all $\xi\in [0,\tau]$ we have
\begin{align*}
\int_\xi^\tau \exp((t-\xi)A^*)  g(t)\, {\rm d}t & = -\int_\xi^\tau \exp((t-\xi)A^*)  
\left[\alpha(t) A^*\eta+\dot\alpha(t)\eta\right]\\
& =-\int_\xi^\tau \frac{\rm d}{{\rm d}t} \left[\alpha(t) \exp((t-\xi)A^*)  \eta\right] \, {\rm d}t= (\tau-\xi)\eta.
\end{align*}
By combining the last formula and \eqref{adj_f_en_bis_bis} it follows that
\begin{equation*}
    \left(\tilde \Psi_\tau g\right)(\sigma)=\frac{(2\sigma-\tau)^2-\tau^2}{2} B^*\eta \boxnote{\qquad \quad \, \, \, \, (\sigma\in [0,\tau]),}
\end{equation*}
which implies $\tilde \Psi_\tau g=0$. We have thus shown by contraposition that, if \eqref{SystemDiffEq_lin} is weakly approximately tracking controllable
in time $\tau$, then $B$ is onto.
\end{proof}

\section{Statement of the main result}\label{section:approxtrackingnonlinear}

In this section, we state and prove our main result (\Cref{theorem:approximatetracking}), which provides weak approximate tracking controllability of \eqref{equation:main_system} in the presence of quadratic nonlinearities.  Moreover, we provide several illustrative examples.

To specify suitable choices for $f$, $U$, and $B$, we borrow from \cite{KoikeNersesyanRisselTucsnak2025,Ners-2015} a saturation assumption that is more general than \Cref{assumption:1}; for further details, see also \cite{AS-2006} and \cite{shirikyan-cmp2006}, as well as the references therein. 

Before stating this more general assumption below (see~\Cref{assumption:2}), we highlight \Cref{assumption:1} as it simplifies the argument presented later, and as it gives already rise to a rich class of examples.

\begin{example}
	There has been interest in types of tracking controllability for the Lorenz system; for instance, see \cite{Gao2019} and the references therein. In this direction, we consider $x = [x_1, x_2, x_3]^{\top} \colon \mathbb{R}^3\longrightarrow\mathbb{R}^3$ as the solution to the controlled nonlinear problem \eqref{equation:main_system} with
	\begin{gather*}
		f(x) \coloneq \begin{bmatrix}
			0  \\
			x_2x_3 \\
			-x_1x_2
		\end{bmatrix}, \quad A \coloneq \begin{bmatrix}
			\sigma & - \sigma & 0  \\
			- \rho & 1 & 0 \\
			0 & 0 & \beta 
		\end{bmatrix}, \quad
		B \coloneq \begin{bmatrix}
			1 & 0  \\
			0 & 1  \\
			0 & 1 
		\end{bmatrix},
	\end{gather*}
	where $\beta, \sigma, \rho > 0$. One can readily check by means of elementary algebraic manipulations that \Cref{assumption:1} holds for the above choices of $f$ and $B$ with the two-dimensional control space $U = \mathbb{R}^2$.
\end{example}

The next definition states a generalized version of \Cref{assumption:1}, allowing us to consider certain high-dimensional systems driven by low-dimensional inputs. For additional context, see \cite{shirikyan-cmp2006} and the references therein, in particular the keyword \enquote{convexification principle}.

\begin{definition}\label{definition:ls}
	Given any subspace $E \subset X$, we denote by $\mathscr{F}(E)$ the largest subspace of $X$ such that any $\gamma \in \mathscr{F}(E)$ can be expressed in the form $\gamma = \xi_0 - \sum_{i=1}^p f(\xi_i)$ for some $p \in \mathbb{N}$ and $\xi_0, \xi_1, \dots, \xi_p \in E$.  
\end{definition}

\begin{assumption}\label{assumption:2}
	For $(E_i)_{i \in \mathbb{N}}$ recursively defined by $E_0 \coloneq \operatorname{Range}(B)$ and $E_i \coloneq  \mathscr{F}(E_{i-1})$ for $i \in \mathbb{N}$, there exists a number $n_X \in \mathbb{N}$ such that $E_{n_X} = X$.
\end{assumption}

The following simple example illustrates that \Cref{assumption:2} is more general than \Cref{assumption:1}.

\begin{example}
    Let $X = \mathbb{R}^6$, $U=\mathbb{R}^3$, and consider the controlled system \eqref{equation:main_system} with $f$ and $B$ chosen as follows:
    \begin{equation*}
		f(x_1,x_2,x_3, x_4,x_5,x_6) \coloneq [0,0,0,x_1x_2, x_1x_3, x_2x_3]^{\top}, \quad B \coloneq [\mathbb{I}_{\mathbb{R}^{3\times3}}, 	\mathbb{O}_{\mathbb{R}^{3\times3}}]^{\top} ,
	\end{equation*}
	where $\mathbb{I}_{\mathbb{R}^{3\times3}}$ denotes the identity matrix in $\mathbb{R}^3$ and $\mathbb{O}_{\mathbb{R}^{3\times3}}$ the $3\times3$ matrix with all entries being zero. It is not difficult to check that \Cref{assumption:2} is satisfied in this case. On the other hand, \Cref{assumption:1} is not satisfied; one can take, for instance, $\gamma = [0,0,0, a, 0, b]^{\top} \in \mathbb{R}^6$ with $a \neq 0 \neq b$ and readily verify that $\gamma = B[u_1,u_2,u_3]^{\top} - f(B [\xi_1,\xi_2, \xi_3]^{\top})$ would imply that both $\xi_3 = 0$ and $\xi_3 \neq 0$, which is impossible.
\end{example}

 We are now in the position to state our main result. 

\begin{theorem}\label{theorem:approximatetracking}
	Let \Cref{assumption:2} be satisfied and $\tau > 0$. Given any $\psi \in W^{1,2}((0,\tau); X)$ and $\varepsilon > 0$, there exists a control $u \in C^{\infty}([0,\tau];U)$ such that the corresponding solution $x$ to~\eqref{equation:main_system} with $x_0 = \psi(0)$ satisfies
	\[
		|x(\tau) - \psi(\tau)| + ||| x-\psi |||_{\tau} < \varepsilon.
	\]
\end{theorem}

 Most of the remaining part of this work is devoted to the proof of the above result (see \Cref{subsection:pa1} for the conclusion of the argument). The strategy of this proof relies on adaptation of arguments in \cite{KoikeNersesyanRisselTucsnak2025}. Some applications of \Cref{theorem:approximatetracking} to coupled systems are discussed in \Cref{{section:dc}}.

\section{An enlarged system}\label{sec_extended}

A classical methodology in nonlinear control theory, introduced in  
Jurdjevic and Kupka \cite{JK81}, called {\em enlargement method}, consists in adding appropriately chosen 
controls and first studying this enlarged system. By ``appropriately" we mean  here  that the enlarged system has to be controllable and that it should be possible to prove that returning to the original system does not essentially
alter its accessibility sets. In this work we apply a version of this methodology 
developed by  \cite{shirikyan-cmp2006} in the continuation of  \cite{AS-05,AS-2006} 
(see also \cite{Ners-2015,KoikeNersesyanRisselTucsnak2025}), in the context of approximate control of the Navier-Stokes or Euler systems. More precisely,  
we introduce the enlarged system
\begin{equation}\label{equation:aux_system}
	\begin{gathered}
		\dot w(t)+ A(w(t) + \zeta(t)) + f(w(t)+\zeta(t)) = \gamma(t),\\
		w(0) = w_0,
	\end{gathered}
\end{equation}
which is driven by two inputs $\zeta, \gamma \colon [0,\tau]\to X$ and reduces to \eqref{equation:main_system} when $\zeta=0$.  
The crucial ingredient  of the proof of \Cref{theorem:approximatetracking} consists in proving that solutions to \eqref{equation:main_system}, driven by an additive input, can be approximated by solutions to \eqref{equation:aux_system}, driven by two inputs; see \Cref{subsection:pa1} below for the details. 

In this section, we limit ourselves to give a consequence of well-known theorems on continuous dependence of solutions of Cauchy problems with respect to parameters (see, for instance, Gamkrelidze \cite[Ch.4]{gamkrelidze2013principles} or \cite[Section 2.6]{barbu2016differential}).  Note that a similar result for PDEs, namely for the $3$D Navier--Stokes system, can be found in \cite[Theorem 1.3]{Ners-2015}  (see also  \cite{Ners-2015}, \cite{shirikyan-cmp2006}). We state this in \Cref{proposition:perturbativestability} below and, for the reader's convenience, sketch the main steps of the proof.  But first, we introduce the following convenient notation to describe solutions to \eqref{equation:aux_system}, whenever they exist, in terms of the imposed data.

\begin{definition}
    Whenever \eqref{equation:aux_system} admits for given $\tau > 0$, $w_0\in X$, $\zeta\in L^2([0,\tau]; X)$, and $\gamma\in L^1([0,\tau];X)$ a unique solution, we denote by
	\[
		\mathcal{R} \colon X\times L^2([0,\tau]; X)\times L^1([0,\tau];X) \longrightarrow W^{1,1}([0,\tau];X)
	\]
    the resolving operator of \eqref{equation:aux_system}. That is, $\mathcal{R}$ maps $(w_0,\zeta,\gamma)\in X\times L^2([0,\tau]; X)\times L^1([0,\tau];X)$ to the unique function~$w = \mathcal{R}(w_0,\zeta,\gamma)$ solving \eqref{equation:aux_system}, as long as such a solution exists.
\end{definition}

We can state now the perturbative well-posedness result on which we will rely extensively in the following sections.

\begin{proposition}\label{proposition:perturbativestability}
	Given $w_0\in X$ and $\zeta \in L^2([0,\tau];X),\ \gamma \in L^1([0,\tau];X)$, assume that \eqref{equation:aux_system} has a solution $w\in W^{1,1}([0,\tau];X)$. Moreover, let $M>0$ be such that
    \begin{equation*}\label{MMARE} 
		\|\zeta\|_{L^2([0,\tau]; X)} + \|\gamma\|_{L^1([0,\tau];X)} + \|w\|_{W^{1,1}([0,\tau];X)}\leqslant M.
	\end{equation*}
    Then, there exist constants $\delta, c > 0$, depending only on $M$ and on $\tau$,
	such that for any $v_0\in X$, $\xi\in L^2([0,\tau]; X)$ and $g\in L^1([0,\tau];X)$ with
	\begin{align}\label{micime}
		|w_0-v_0|+\|\zeta - \xi\|_{L^2([0,\tau]; X)}+\|\gamma - g\|_{L^1([0,\tau]; X)}\leqslant \delta
	\end{align}
	the problem \eqref{equation:aux_system} with initial state $w_0$ replaced by $v_0$ and controls $(\zeta, \gamma)$ replaced by $(\xi, g)$ admits a unique solution $v = \mathcal{R}(v_0,\xi, g)$ on $[0,\tau]$ satisfying
	\begin{equation}\label{equation:pe}
		\begin{aligned}
			& \|\mathcal{R}(w_0,\zeta, \gamma)-\mathcal{R}(v_0,\xi, g)\|_{C([0,\tau]; X)} \\
			& \quad \leqslant c\, \left( |w_0-v_0|+\|\zeta-\xi\|_{L^2([0,\tau]; X)}+\|\gamma-g\|_{L^1([0,\tau]; X)} \right).
		\end{aligned}
	\end{equation}
\end{proposition}
\begin{proof}
	
{\em The first step} consists of applying \cite[Theorem 4.4]{gamkrelidze2013principles}. According to the above quoted result, for every $\varepsilon \in (0,1)$ there exists $\delta>0$, possibly depending on~$M$ and~$\tau$, such that for every $v_0$, $\xi$, and $g$ satisfying \eqref{micime} the system \eqref{equation:aux_system}, with $w_0$ replaced by $v_0$ and  $(\zeta, \gamma)$ replaced by $(\xi, g)$, admits a unique solution $v = \mathcal{R}(v_0,\xi, g)$ that is defined on $[0,\tau]$ and satisfies
\begin{equation}\label{equation:pe_eps}
			 \|\mathcal{R}(w_0,\zeta, \gamma)-\mathcal{R}(v_0,\xi, g)\|_{C([0,\tau]; X)} \leqslant \varepsilon.
	\end{equation}

{\em Second step.} With $\delta$ chosen at the first step, we note that $r$ defined by $r=w-v$ satisfies on $[0,\tau]$ the equation
	\begin{multline}\label{equation:remainder}
			\dot{r} + Ar + \Gamma(w+\zeta,r)+\Gamma (r,v+\xi) \\= A (\zeta-\xi) + \gamma - g
            -\Gamma(w+\zeta,\zeta-\xi)-\Gamma(\zeta-\xi,v+\xi),
	\end{multline}
where we recall from \Cref{section:introduction} that $\Gamma\colon X\times X\to X$ is the symmetric bilinear function associated to $f$; namely, $f(x)=\Gamma(x,x)$ for $x\in X$.
For $t\in [0,\tau]$ we define $\mathcal{A}:[0,\tau]\to \mathcal{L}(X)$ by setting $\mathcal{A}(t)x \coloneq Ax + \Gamma(w+\zeta,x)+\Gamma (x,v+\xi)$ for all $t\in [0,\tau]$ and $x\in X$.
Moreover, we define  $U\colon[0,\tau]\to \mathcal{L}(X)$ as the solution of the initial value problem
\[
    \begin{cases}
        \dot U(t)=-\mathcal{A}(t) U(t),  \boxnote{\qquad\quad \,(t\in(0,\tau)),}\\
        U(0)=\mathbb{I}_X.
    \end{cases}
\]
With the above notation, the equation \eqref{equation:remainder} and the variation of constants formula yield that
\begin{equation}\label{Duhamel}
r(t)=U(t)(w_0-v_0)+\int_0^t U(t)U^{-1}(\sigma) \alpha(\sigma)\, {\rm d}\sigma \boxnote{\quad(t\in [0,\tau]),} 
\end{equation}
where the function $\alpha\colon[0,\tau]\to X$ is given by
\begin{equation}\label{def_alfa}
	\alpha \coloneq A (\zeta-\xi) + \gamma - g
	-\Gamma(w+\zeta,\zeta-\xi)-\Gamma(\zeta-\xi,v+\xi).
\end{equation}
It is not difficult to check (using \eqref{equation:pe_eps}) that, with $\delta$ chosen in the first step, the maps $t\mapsto \|U(t)\|_{\mathcal{L}(X)}$ and $t\mapsto \|U^{-1}(t)\|_{\mathcal{L}(X)}$ are bounded on $[0,\tau]$ by constants depending only on $M$ and $\tau$. This fact and \eqref{Duhamel} imply that there exists a constant $c_1 > 0$ (depending only on $M$ and $\tau$) such that
\begin{equation}\label{Duhamel_ineq}
\|r(t)\|_X \leqslant c_1\left(\|w_0-v_0\|_X+\int_0^t \|\alpha(\sigma)\|_X \, {\rm d}\sigma\right) \boxnote{\quad \,\,\,(t\in [0,\tau]).}
\end{equation}
On the other hand, from \eqref{equation:pe_eps} and \eqref{def_alfa} it follows that there exists $c_2>0$ (depending only on $M$ and $\tau$) such that
\[
\|\alpha\|_{L^1([0,t];X)}\leqslant c_2 \left(\|\zeta-\xi\|_{L^2([0,\tau];X)}+\|\gamma-g\|_{L^1([0,\tau];X)}\right)
\]
\[
	 \boxnote{\qquad \qquad \qquad \qquad \qquad \quad \, (t\in [0,\tau]).}
\]
The above estimate and \eqref{Duhamel_ineq} yield the conclusion \eqref{equation:pe}.

\end{proof}

\section{An approximation result}\label{new_sec_approx}

We denote the vector spaces $(E_l)_{l\in\mathbb{N}}$ constructed as in \Cref{assumption:2}. The main result of this section asserts under an assumption on $f$ that for $l\in \mathbb{N}$, the solutions to \eqref{equation:aux_system} with $\gamma$ taking values in $E_l$ and $\zeta=0$ can be approximated by a sequence of solutions to the same system driven by a sequence of additive inputs $(\gamma_n)_{n \in \mathbb{N}}$ with values in $E_{l-1}$ and a sequence of multiplicative inputs $(\zeta_n)_{n \in \mathbb{N}}$ with values in $E_{l-1}$ that vanishes asymptotically in the relaxation norm. Such approximation properties are known for incompressible fluids controlled through a finite number of Fourier modes (see \cite{Ners-2015,KoikeNersesyanRisselTucsnak2025}) and we adapt these techniques to the present setting. More precisely, we have:

\begin{proposition}\label{proposition_pwc}
	Let $l \in \mathbb{N}$, $\gamma \colon [0,\tau]\to E_l$ be a piecewise constant function, and $w_0\in X$. Moreover, assume that the trajectory $\mathcal{R}(w_0,0,\gamma)$ is well-defined on $[0,\tau]$ and that any $\gamma \in E_l$ is of the form $\gamma = \xi_0 - \sum_{i=1}^p f(\xi_i)$ for some $p \in \mathbb{N}$ and $\xi_0, \xi_1, \dots, \xi_p \in E_{l-1}$. Then, there is a sequence $\{(\gamma_n,\zeta_ n)\}_{n\in \mathbb{N}} \subset C^{\infty}([0,\tau];E_{l-1}\times E_{l-1})$ such that 
	\begin{equation*} 
		\sup_{n\in \mathbb{N}}\left(\| \zeta_n\|_{C([0,\tau]; X)}+\|\gamma_n\|_{L^2([0,\tau]; X)}\right) < \infty
	\end{equation*}
	and
	\begin{equation*}
		\| \mathcal{R}(w_0,0,\gamma)- \mathcal{R}(w_0, \zeta_n, \gamma_n)\|_{C([0,\tau]; X)}+  |||\zeta_n|||_{\tau} \to 0  
	\end{equation*}
	as $n\to\infty$.
\end{proposition}

To prove the above result we need some notation and an elementary lemma. Given $n \in \mathbb{N}$, we denote by $\mathscr{E}_n\colon X \longrightarrow L^2((0,\tau); X)$ the operator which assigns to $\xi \in X$ the function
\begin{equation}\label{DEFEN}
(\mathscr{E}_n \xi)(t) \coloneq \zeta(nt/\tau) \boxnote{\qquad\qquad(t\in [0,\tau]),}
\end{equation}
where~$\zeta(t)$ denotes the~$1$-periodic continuation of $s \mapsto (\mathbbm{1}_{\left[0,1/2\right)}(s) - \mathbbm{1}_{\left[1/2,1\right)}(s))\xi$ and $\mathbbm{1}_I$ stands for the indicator function of an interval $I$.
Moreover, let $\mathscr{K}$ be the integral operator defined via
\begin{equation}\label{defkappa}
(\mathscr{K}z)(t) \coloneqq \int_0^t z(s) \, {\rm d}s
\end{equation}
for $z\in L^1([0,\tau]; X)$ and $t \in [0,\tau]$.

\begin{lemma}\label{lemma:convergenceKgn}
	Let $\phi \colon X \longrightarrow X$ be continuous. For $\xi \in X$ and $w_1 \in C([0,\tau]; X)$, define $\zeta_n \coloneq \mathscr{E}_n \xi$ and denote the sequence
	\begin{equation*}
		h_n\coloneqq \phi(w_1+\zeta_n)-\frac12\left( \phi(w_1+\xi)+\phi(w_1-\xi)\right) + A\zeta_n
	\end{equation*}
	for all $n \in \mathbb{N}$. Then,
	\[
		\displaystyle\lim_{n\to\infty}\|\mathscr{K} h_n\|_{C([0,\tau];X)} = 0.
	\]
\end{lemma}
\begin{proof}
	We first note that the general statement follows by means of an approximation argument from the case where the function $w_1 \colon [0,\tau]\to X$ is piecewise constant. Next, we note that
	the family $\{\mathscr{K} h_n\}_{n\in \mathbb{N}}$ is relatively compact in~$C([0,\tau];X)$ for piecewise constant choices of $w_1$. Indeed, $\{h_n(t)\}_{t \in [0,\tau]}$ is contained in a finite set,
	not depending on $n$.     
  	Thus, there exists a
	compact set $G \subset X$ such that
	$(\mathscr{K} h_n)(t) \in G$ for all $t \in [0,\tau]$ and $n \in \mathbb{N}$.
	Moreover, because of $\sup_{n \in \mathbb{N}}\| h_n\|_{C([0,\tau]; X)} < \infty$, the family $\{\mathscr{K} h_n\}_{n\in \mathbb{N}}$ is uniformly equicontinuous on $[0,\tau]$; hence, $\{\mathscr{K} h_n\}_{n\in \mathbb{N}}$ is
	relatively compact in $C([0,\tau]; X)$ by the Arzel\`{a}--Ascoli theorem. Therefore, it suffices to establish for all $t \in [0,\tau]$ the convergence
	\begin{equation}\label{equation:limitKgn}
		\|\mathscr{K} h_n(t)\|_{X} \to 0 \, \, \text{as} \,\,
		n\rightarrow \infty.
	\end{equation}
	
	In a first step, the convergence in \eqref{equation:limitKgn} is demonstrated for the case where $w_1 \equiv b$ with a constant $b \in X$. Hereto, any fixed $t \in [0,\tau]$ is decomposed for each $n \in \mathbb{N}$ as $t = t_{l, n} + \tau_n$, where
	\begin{equation}\label{equation:reppar}
		t_{l,n} = \frac{lT}{n}, \quad l = l_{t,n} \in \mathbb{N}, \quad \tau_n\in \left[0,\frac{T}{n}\right).
	\end{equation}
	Due to the definition of $(\zeta_n)_{n \in \mathbb{N}}$ and using \eqref{DEFEN} it follows that
	\begin{gather*}
		\int_0^{\frac{lT}{n}} A\zeta_n(s) \, {\rm d}s = A \int_0^{\frac{lT}{n}} \zeta_n(s) \, {\rm d}s = 0,\\
		\int_0^{\frac{lT}{n}} \phi(b +
		\zeta_n(s)) \, {\rm d}s = \frac{lT}{2n}\left(\phi(b+\xi)+ \phi(b-\xi)\right).
	\end{gather*}
	Consequently,
	\begin{align*}
		\int_0^{\frac{lT}{n}} h_n (s) \, {\rm d}s
		& = \int_0^{\frac{lT}{n}} \left[\phi(b +
		\zeta_n(s)) + 
		A\zeta_n(s)\right] \, {\rm d}s \\ 
		& \quad -\frac{lT}{2n}\left(\phi(b+\xi)+ \phi(b-\xi)\right) \\
		& = 0,
	\end{align*}
	which yields
	\begin{equation*}
		\mathscr{K} h_n(t) = \int_0^{\tau_n}
		\left[\phi(b + \zeta_n(s)) + A\zeta_n(s)\right] \, {\rm d}s -\frac{\tau_n}{2} \left(  \phi (b+\xi) + \phi(b-\xi)\right).
	\end{equation*}
	Since $\tau_n \to 0 $ as $n \to \infty$, one can first conclude the convergence \eqref{equation:limitKgn} for constant~$w_1$ and then by a similar analysis obtain
	\eqref{equation:limitKgn} also for the case where $w_1$ is piecewise constant.
\end{proof}

We are now in a position to prove the main result of this section. 

	\begin{proof}[Proof of \Cref{proposition_pwc}]
		It can be assumed that~$\gamma\in X$ is a constant function of time. Otherwise, one may apply the strategy below on each interval of constancy and then use \Cref{proposition:perturbativestability} for obtaining smooth controls. Moreover, we simplify the presentation by first making the additional assumption that there exist elements $\eta, \xi \in E_{l-1}$ such that
		\begin{equation}\label{equation:gammauxi}
			\gamma = \eta - f(\xi).
		\end{equation}
		The general situation is treated almost analogously, as discussed in \Cref{remark:general} below. 
		
		Taking the quadratic character \eqref{equation:qc} into account, one obtains for each $w \in X$ the representation
		\[
		f(w)-\gamma=\frac12 \left( f(w+\xi)+f(w-\xi)\right)- \eta.
		\]
		Thus, the trajectory $w_1 \coloneqq {\mathcal{R}}(w_0,0,\gamma)\in C([0,\tau];X)$
		satisfies 
		\begin{equation}\label{equation:w1}
			\dot{w}_1 + Aw_1 + \frac12\left(f(w_1+\xi)+f(w_1-\xi)\right) = \eta.
		\end{equation}
		Then, we introduce for each $n \in \mathbb{N}$ the function $\zeta_n(t) \coloneq \mathscr{E}_n \xi$,
        where $\mathscr{E}_n$ has been introduced in \eqref{DEFEN}.
		In particular, the equation \eqref{equation:w1} is equivalent to  
		\begin{equation}\label{equation:w1e}
			\dot w_1 + A(w_1+\zeta_n) + f(w_1 + \zeta_n) = \eta
			+ g_n,
		\end{equation}
		where
		\begin{equation}\label{equation:gn}
			g_n \coloneqq f(w_1+\zeta_n)-\frac12\left( f(w_1+\xi)+f(w_1-\xi)\right) + A\zeta_n.
		\end{equation}
		The functions $(g_n)_{n \in \mathbb{N}}$ in the right-hand side of \eqref{equation:w1e} are hereby undesired, as they may assume values not contained in $E_{l-1}$. This issue will be avoided by using \Cref{lemma:convergenceKgn}.

		We consider now for each $n \in \mathbb{N}$ the transformed trajectory $v_n \coloneqq w_1-\mathscr{K}g_n$,
        with $\mathscr{K}$ introduced in \eqref{defkappa}. Then, we have that $v_n$ satisfies for all $n \in \mathbb{N}$ the system
		\begin{equation*}\label{eq_vn}
				\dot v_n + A(v_n + \zeta_n + \mathscr{K} g_n)  + f (v_n + \zeta_n + \mathscr{K} g_n) = \eta
		\end{equation*}
		with initial value $v_n(0) = w_0$, so that
		\begin{equation}\label{equation:vnR}
			v_n = \mathcal{R}(w_0,\zeta_n+\mathscr{K}g_n, \eta).
		\end{equation}
		In \eqref{equation:vnR}, the right-hand side is well-defined for each $n \in \mathbb{N}$, noting that the left-hand side equals $w_1-\mathscr{K}g_n$. Resorting to \Cref{lemma:convergenceKgn}, it follows that
		\begin{equation}\label{equation:supvn}
			\sup_{n\in \mathbb{N}} \|v_n \|_{C([0,\tau];X)} < \infty.
		\end{equation}
		Moreover, the definition of $\zeta_n$ implies that  $\zeta_n(t)\in \{\xi,-\xi\}$ for every $n\in\mathbb{N}$ and $t\in [0,\tau]$. Consequently, 
		\begin{equation}\label{equation:finitenormzeta}
			\sup_{n \in \mathbb{N}} \| \zeta_n\|_{L^\infty([0,\tau];E_{l-1})} < \infty.
		\end{equation}
		Hence, by \eqref{equation:supvn}, \eqref{equation:finitenormzeta}, \Cref{proposition:perturbativestability}, and~\Cref{lemma:convergenceKgn} we know that the solution ${\mathcal{R}}(w_0,\zeta_n,\eta)$ is well-defined for sufficiently large $n$, and it holds that 
		\begin{equation}\label{equation:vnl}
			\|{\mathcal{R}}(w_0,\zeta_n,\eta)-v_n \|_{C([0,\tau];X)}\rightarrow 0 \, \, \text{as} \,\,
			n\rightarrow \infty.
		\end{equation}     
		As a result, it follows that
		\begin{equation*}\label{eq:Step1_conclusion}
			\begin{aligned}
				\|{\mathcal{R}}(w_0,\zeta_n,\eta)-w_1 \|_{C([0,\tau];X)} \longrightarrow 0 \, \, \text{as} \,\,
				n\rightarrow \infty,
			\end{aligned}
		\end{equation*}
		where we also used \eqref{equation:vnl}, \Cref{lemma:convergenceKgn}, and the fact that $v_n = w_1-\mathscr{K}g_n$ for all $n \in \mathbb{N}$.
		
		To conclude the proof of \Cref{proposition_pwc}, remains to show the convergence in relaxation metric
		\begin{equation}\label{equation:auxrellim}
			||| \zeta_ n |||_{\tau} \to 0 \, \, \text{ as } n \to \infty.
		\end{equation} 
		This is done by analysis similar to the proof of \Cref{lemma:convergenceKgn}. In fact, it suffices to verify that $\{\mathscr{K}\zeta_ n\}_{n\in \mathbb{N}}$ is relatively compact in $C([0,\tau]; X)$ and that $\|(\mathscr{K}\zeta_n)(t)\|_{X}\to 0$ as $n\to\infty$ for any $t\in [0,\tau]$. To this end, we remark that the set $\{\zeta_n(t) \, | \, t \in [0,\tau], n \in \mathbb{N}\}$ is contained in a finite
		subset of~$X$ independent of $n$, implying that $(\mathscr{K}\zeta_n)(t)$ belongs for each $t \in [0,\tau]$ and $n \in \mathbb{N}$ to a fixed compact set in~$X$.
		In particular, one can infer from~\eqref{equation:finitenormzeta} that $(\mathscr{K}\zeta_n)_{n\in \mathbb{N}}$ is uniformly equicontinuous on $[0,\tau]$ such that $\{K\zeta_n\}_{n\in \mathbb{N}}$ is
		relatively compact in~$C([0,\tau]; X)$ by the Arzel\`{a}--Ascoli theorem. 
		Representing again any $t \in [0,\tau]$ as $t=t_{l,n}+\tau_n$, where $l,n, \tau_n$ are as in \eqref{equation:reppar}, and by using the fact that $(\mathscr{K}\zeta _n)(lT/n)=0$, we obtain via \eqref{equation:finitenormzeta} the limit $\|(\mathscr{K}\zeta_n)(t)\|_{X}\to 0$ as $n\to\infty$. This completes the proof of the limit~\eqref{equation:auxrellim}.

		Finally, due to \Cref{proposition:perturbativestability}, the sequence $(\zeta_n)_{n \in \mathbb{N}}$ can be replaced by a sequence of smooth functions.
	\end{proof}

	\begin{remark}\label{remark:general}
		We explain how to obtain the general assertion of \Cref{proposition_pwc}, where instead of \eqref{equation:gammauxi} one only knows from the definition of $E_l$ that $\gamma = \xi_0 - \sum_{i=1}^p f(\xi_i)$
		for some $p \in \mathbb{N}$ and $\xi_0, \xi_1, \dots, \xi_p \in E_{l-1}$. Given any $w \in E_l$, and defining $m \coloneq 2p$, it then holds
		\[
		f(w) - \gamma = \frac{1}{m} \sum_{i = 1}^m f(w+\zeta^i) - \xi_0,
		\]
		where $\zeta^i = (m/2)^{1/2} \xi_i$ and $\zeta^{i+p} = -(m/2)^{1/2} \xi_i$ for $i \in \{1, \dots, p\}$. In contrast to the argument given above when \eqref{equation:gammauxi} is known, the sequence $(\zeta_n)_{n \in \mathbb{N}}$ is now defined by fixing the $1$-periodic function $\zeta \colon \mathbb{R} \longrightarrow E_{l-1}$ such that $\zeta(t) = \zeta^i$ for $t \in [(i-1)/m, i/m]$ and $i \in \{1, \dots, m\}$, followed by setting $\zeta_n(t) = \zeta(nt/\tau)$ for each $n \in \mathbb{N}$. Then, the trajectory $w_1 \coloneqq {\mathcal{R}}(w_0,0,\gamma)\in C([0,\tau];X)$ satisfies the problem
		\[
			\dot{w}_1 + A(w_1 + \zeta_n) + f(w_1 + \zeta_n) = \xi_0 + h_n,
		\]
		where
		\begin{equation}\label{equation:gn2}
			h_n(t) \coloneqq f(w_1+\zeta_n)-\frac{1}{m} \sum_{i=1}^m f(w_1+\zeta^i) + AB\zeta_n.
		\end{equation}
		Considering $h_n$ defined in \eqref{equation:gn2} instead of $g_n$ as given in \eqref{equation:gn}, the formulation and proof of \Cref{lemma:convergenceKgn} immediately extends to this situation, and one can follow with minor adaptations the argument given above. 
	\end{remark}

	\section{Proof of the main result}\label{subsection:pa1}
    
	We are now in the position to complete the proof of \Cref{theorem:approximatetracking}. Hereto, attention is paid first to the simplified situation of \Cref{theorem:approximatetracking_simp}, where \Cref{assumption:1} holds, as this emphasizes the core argument. 
	
	\begin{proof}[Proof of \Cref{theorem:approximatetracking} and of \Cref{theorem:approximatetracking_simp}]
		In a {\em first step}, we suppose \Cref{assumption:1}, while denoting the spaces $E_1 = \operatorname{Range}(B)$ and $E_2 = X$ in order to be consistent with the notation used in \Cref{assumption:2}. That is, we first prove  \Cref{theorem:approximatetracking_simp}.
        
        Let $\psi \in W^{1,2}((0,\tau); X)$ be the target trajectory, set $w_0 \coloneq \psi(0)$, and define an associated control by $\gamma \coloneq \dot{\psi} + A\psi + f(\psi) \in L^2((0,\tau); X)$. Then, one has $\psi = \mathscr{R}(w_0, 0, \gamma)$.
		As the goal is to approximate $\psi$, and in view of \Cref{proposition:perturbativestability}, we can assume without loss of generality that $\gamma$ is piecewise constant. Now, we associate with $\psi$ and $\gamma$ a sequence $(\zeta_n, \gamma_n)_{n\in\mathbb{N}}\subset C^{\infty}([0,\tau];E_{1}\times E_{1})$ in the sense of \Cref{proposition_pwc}.
        Again, due to \Cref{proposition:perturbativestability} and an approximation argument, there is a sequence $(\widetilde{\zeta}_n)_{n\in\mathbb{N}} \subset C^{\infty}_0((0,\tau);E_1)$, which is uniformly bounded in $C([0,\tau]; E_1)$, such that
		\begin{equation}\label{equation:zetanzetantildediff}
			\begin{gathered}
				\lim\limits_{n \to \infty} \|\zeta_n - \widetilde{\zeta}_n\|_{L^2((0,\tau); E_1)} = 0,\\
				\lim\limits_{n \to \infty} \|\mathscr{R}(w_0, \zeta_n, \gamma_n) - \mathscr{R}(w_0, \widetilde{\zeta}_n, \gamma_n)\|_{C([0,\tau]; X)} = 0.
			\end{gathered}
		\end{equation}
		Noting that the sequence $(\zeta_n, \gamma_n)_{n\in\mathbb{N}}$ is obtained via \Cref{proposition_pwc}, we get
		\begin{equation}\label{equation:relaxationnormzetantilde}
			|||\widetilde{\zeta}_n|||_{\tau} \leq \int_0^\tau |\zeta_n(s) - \widetilde{\zeta}_n(s)| \, {\rm d}s + ||| \zeta_n |||_{\tau} \longrightarrow 0
		\end{equation}
		as $n \longrightarrow \infty$. Now, we decompose
		\begin{equation}\label{equation:decomposeRzetanzetantilde}
			\mathscr{R}(w_0, \widetilde{\zeta}_n, \gamma_n) = \mathscr{R}(w_0, 0, \widetilde{\gamma}_n) - \widetilde{\zeta}_n,
		\end{equation}
		where $\widetilde{\gamma}_n \coloneq \gamma_n + {\rm d} \widetilde{\zeta}_n/{{\rm d}t}$.
		In particular,
		\begin{equation}\label{equation:zetanzetantildeRendpoint}
			\mathscr{R}_{\tau}(w_0, \widetilde{\zeta}_n, \gamma_n) = \mathscr{R}_{\tau}(w_0, 0, \widetilde{\gamma}_n)
		\end{equation}
		by the choice of $(\widetilde{\zeta}_n)_{n\in\mathbb{N}}$. Then, thanks to \Cref{proposition_pwc}, \eqref{equation:zetanzetantildediff}, and \eqref{equation:zetanzetantildeRendpoint}, it follows that $\mathscr{R}_{\tau}(w_0, 0, \widetilde{\gamma}_n) \longrightarrow \mathscr{R}_{\tau}(w_0, 0, \gamma)$ as $n \longrightarrow \infty$. Further, due to \Cref{proposition_pwc} and \cref{equation:zetanzetantildediff,equation:relaxationnormzetantilde,equation:decomposeRzetanzetantilde,equation:zetanzetantildeRendpoint}, the norm $|||\mathscr{R}(w_0, 0, \gamma) -  \mathscr{R}(w_0, 0, \widetilde{\gamma}_n) |||_{\tau}$ is for each $n \in \mathbb{N}$ bounded from above by
		\begin{equation*}
			\begin{multlined}
				|||\mathscr{R}(w_0, \widetilde{\zeta}_n, \gamma_n) - \mathscr{R}(w_0, 0, \widetilde{\gamma}_n)|||_{\tau} 
				+  \tau \|\mathscr{R}(w_0, 0, \gamma) - \mathscr{R}(w_0, \zeta_n, \gamma_n)\|_{C([0,\tau]; X)}  \\
			 	+ \tau \|\mathscr{R}(w_0, \zeta_n, \gamma_n) - \mathscr{R}(w_0, \widetilde{\zeta}_n, \gamma_n)\|_{C([0,\tau];X)}.
			\end{multlined}
		\end{equation*}
		As a result, we obtain the limit
		\[
			||| \mathscr{R}(w_0, 0, \gamma) -  \mathscr{R}(w_0, 0, \widetilde{\gamma}_n) |||_{\tau} \longrightarrow 0 \mbox{ as } n \longrightarrow \infty,
		\]
		which implies that the desired control $u \in C^{\infty}([0,\tau]; U)$ can for sufficiently large $\widetilde{N} \in \mathbb{N}$ be chosen of the form $u = \widetilde{\gamma}_{\widetilde{N}}$
	
		The {\em second step} is to consider the general situation where \Cref{assumption:1} is replaced by \Cref{assumption:2}. Hereto, we denote by $n_X$ the number introduced in \Cref{assumption:2}. Next, let $\varepsilon > 0$, $\psi \in W^{1,2}((0,\tau); X)$, and $w_0 \coloneq \psi(0)$. Then, define the auxiliary control $\gamma \coloneq A\psi + \dot{\psi} + f(\psi)$ and take $\varepsilon_1, \dots, \varepsilon_{n_X} \in (0,\infty)$ such that $\sum_{l=1}^{n_X} \varepsilon_l < \varepsilon$. In particular, it holds that $\psi = \mathscr{R}(w_0, 0, \gamma)$. Now, the same idea as described above under \Cref{assumption:1} is utilized to obtain a suitable $E_{n_X-1}$-valued control $\widetilde{\gamma}_{1}$ so that
		\[
			||| \mathscr{R}(w_0, 0, \gamma) -  \mathscr{R}(w_0, 0, \widetilde{\gamma}_{1}) |||_{\tau} < \varepsilon_1.
		\]
		If necessary, this argument can be repeated to obtain a suitable $E_{n_X-2}$-valued control $\widetilde{\gamma}_{2}$ such that
		\[
			||| \mathscr{R}(w_0, 0, \widetilde{\gamma}_{1}) -  \mathscr{R}(w_0, 0, 	\widetilde{\gamma}_{2}) |||_{\tau} < \varepsilon_2.
		\]
		After iterating this procedure a finite number of time, one arrives at an $E_1$-valued control $\widetilde{\gamma}$ with $||| \mathscr{R}(w_0, 0, \gamma) -  \mathscr{R}(w_0, 0, \widetilde{\gamma}) |||_{\tau} < \varepsilon$.
	\end{proof}

	\section{Applications to the control of coupled systems}\label{section:dc}
    
	In this section, we apply \Cref{theorem:approximatetracking}  to the study of controllability problems for systems which couple \eqref{equation:main_system}
    with an ODE system where the state $x$ in \eqref{equation:main_system} plays the role of a coefficient. More precisely, 
    let $Z$ be another finite-dimensional inner product space, let $\widetilde{\Gamma}\colon Z \times X \longrightarrow Z$ be bilinear, and let $F\colon Z \to Z$ be locally Lipschitz continuous.
    Then, using the notations from \eqref{equation:main_system}, we consider the coupled system
	\begin{equation}\label{equation:coupled}
		\begin{cases}
			\dot{x}(t) + A x(t) + f(x(t)) = Bu(t) \boxnote{\qquad\qquad(t\geqslant 0),} \\
			\dot{z}(t) + \tilde\Gamma(z(t), x(t)) + F(z(t)) = 0 \boxnote{\qquad \quad \, (t\geqslant 0),}\\
			x(0)=x_0\in X, z(0) = z_0 \in  Z,
		\end{cases} 
	\end{equation}
    where the function $u \colon [0,\infty) \longrightarrow U$ is the control, $A \in \mathcal{L}(X)$ and $B \in \mathcal{L}(U, X)$ are linear maps from $X$ to $X$ and from $U$ to $X$, respectively, and $f\colon X\longrightarrow X$ is of the form $f(x)=\Gamma(x,x)$ for a bilinear function $\Gamma \colon X\times X\to X$.

	The following theorem provides sufficient conditions for the approximate tracking controllability of~$z$ with respect to the uniform norm, the target being any fixed trajectory of the second equation in \eqref{equation:coupled}, via a control acting in the equation for~$x$. Simultaneously, it provides the approximate controllability of~the final state~$x(\tau)$. More precisely, we have: 
	
	\begin{theorem}\label{theorem:application}
		Let $\tau > 0$. Assume that $(\overline{z}, \overline{x}) \in W^{1,2}((0,\tau); Z)\times C([0,\tau]; X)$ with initial values $(\overline{z}(0), \overline{x}(0)) = (z_0, x_0)$ satisfy on the time interval $[0,\tau]$ the equation
		\begin{equation}\label{equation:reference}
			\begin{gathered}
				\dot{\overline{z}}(t)  + \widetilde{\Gamma}(\overline{z}(t), \overline{x}(t)) + F(\overline{z}(t)) = 0.
			\end{gathered}
		\end{equation}
		Moreover, suppose that $f, B$, and the control space $U$ in \eqref{equation:coupled} satisfy \Cref{assumption:2}. Then, for any $\varepsilon > 0$, there exists a control $u \in C^{\infty}([0,\tau]; U)$ such that the solution to \eqref{equation:coupled} obeys
		\[
			\sup\limits_{t \in [0,\tau]} |z(t) - \overline{z}(t)| + |x(\tau) - \overline{x}(\tau)| + |||x - \overline{x}|||_{\tau} < \varepsilon.
		\]
	\end{theorem}

The proof of the above result provides a finite-dimensional illustration of a mechanism, exploited recently in \cite{KoikeNersesyanRisselTucsnak2025} (in a nonlinear PDE context) for the study of relaxation enhancement of concentrations immersed in controlled fluids, and in \cite{Ners-2015} for the Lagrangian controllability of the Navier--Stokes system. For the here considered finite-dimensional setup, we extend these arguments further; namely, here we allow $F$ to be a general locally Lipschitz function instead of requiring it to be linear. Before giving the proof of \Cref{theorem:application}, we state some remarks and examples.
	
	\begin{remark}[Dynamic control]\label{example:dc}
The system described by the first equation in \eqref{equation:coupled} can be seen as a dynamic controller of the system
     (with state~$z$) described by the second one. This means that if the system (with state trajectory~$z$ and input signal~$x$) described by the second equation in \eqref{equation:coupled} is approximately controllable in time~$\tau$, then  \Cref{theorem:application} implies that this system can be controlled using a dynamic controller described by the first equation in \eqref{theorem:application}. More precisely, for every $z_1\in Z$ and $\varepsilon>0$ there exists $u\in L^2([0,\tau];U)$ such that the solution $(x,z)$ to \eqref{equation:coupled} satisfies $\|z(\tau)-z_1\|_Z < \varepsilon$, and one has even the stronger statement that $u$ can be chosen such that $\sup_{t\in[0,\tau]}\|z(t)- \overline{z}(t)\|_Z < \varepsilon$, where $(\overline{z}, \overline{x})$ is any sufficiently regular pair satisfying \eqref{equation:reference} with $\overline{z}(0) = z_0$ and $\|\overline{z}(\tau)-z_1\|_Z < \varepsilon$. Further, it is possible to simultaneously drive $x$ approximately to any prescribed target state at the final time $\tau$.
	\end{remark}  

	\begin{example}[Dynamic motion planning]\label{example:dmp}
		When $X=Z=\mathbb{R}^d$, $d \geq 1$, and the bilinear function $\widetilde{\Gamma}$ in the second equation of \eqref{equation:coupled}
        is $\widetilde{\Gamma}(a,b) \coloneq [a_1b_1, a_2b_2,\dots,a_db_d]^{\top}$ for $a,b \in \mathbb{R}^d$, we can maintain the trajectory of $z$ in the uniform norm close to a given curve by using the dynamic controller $x$. In this example, the aim is to track with $z$ any prescribed reference curve $z_{\operatorname{ref}} \in C([0,\tau]; Z)$ that stays within an orthant (the components of $z_{\operatorname{ref}}$ do not change sign),  
        More precisely, for any $\tau > 0$ we consider the coupled system 
		\begin{equation}\label{equation:coupled2ex}
			\begin{cases}
				\dot{x}(t) + A x(t) + f(x(t)) = Bu(t),  \boxnote{\qquad\qquad(t\geqslant 0),}\\
				\dot{z}(t) + \widetilde{\Gamma}(z(t),x(t)) + F(z(t)) = 0, \boxnote{\qquad \,\,\,\,\,\, \,(t\geqslant 0),}\\
				(x,z)(0) = (x_0, z_0) \in X\times Z,
			\end{cases}
		\end{equation}
		where $f, B$, and the control space $U$ satisfy \Cref{assumption:2}, and $F\colon Z \to Z$ is locally Lipschitz continuous.
		Given $\varepsilon,\ \tau > 0$, we fix a smooth curve $\overline{z} = [\overline{z}_1, \overline{z}_2, \dots, \overline{z}_d]^{\top}\colon[0,\tau]\longrightarrow\mathbb{R}^d$ such that $\Pi_{i=1}^d\overline{z}_i(r) \neq 0$ for all $r \in [0,\tau]$ and $\sup_{t \in [0,\tau]} |z_{\operatorname{ref}}(t) - \overline{z}(t)| < \varepsilon/2$. Then, for each $t \in [0,\tau]$ and $l \in \{1,2,\dots,d\}$, we define
		\[
			\overline{x}_l(t) \coloneq \frac{F_l(\overline{z}(t)) - \dot{\overline{z}}_l(t)}{\overline{z}_l(t)}.
		\]
		As a result, $(\overline{z}, \overline{x})$ together with $f, B$, and the control space $U$, satisfy the assumptions of \Cref{theorem:application}. Now, appealing to \Cref{theorem:application} with $\varepsilon$ replaced by $\varepsilon/2$, there exists a control $u \in L^2((0,\tau); U)$ such that the solution to \eqref{equation:coupled2ex} obeys
		\[
			\sup\limits_{t \in [0,\tau]} |z(t) - z_{\operatorname{ref}}(t)| \leq \sup\limits_{t \in [0,\tau]} \left(|z(t) - \overline{z}(t)| + |z_{\operatorname{ref}}(t) - \overline{z}(t)|\right) < \varepsilon.
		\]
	\end{example}

	The proof of \Cref{theorem:application} is a consequence of \Cref{theorem:approximatetracking}, providing controlled trajectories bounded in $C([0,\tau]; X)$ independently of the appearing parameter $\varepsilon$ (see \eqref{equation:decomposeRzetanzetantilde}), and the following adaptation of \cite[Theorem 3.2]{KoikeNersesyanRisselTucsnak2025}, which we generalize here further to allow the presence of any locally Lipschitz continuous $F$.

	\begin{lemma}\label{lemma:app}
		Let $\tau > 0$ and for any $z_0 \in Z$ and $\widetilde{x} \in C([0,\tau]; X)$ assume that on the time interval $[0, \tau]$ there exists a unique solution $\widetilde{z}$ to
		\begin{gather*}
			\dot{\widetilde{z}}(t) + \widetilde{\Gamma}(\widetilde{z}(t), \widetilde{x}(t)) + F(\widetilde{z}(t)) = 0, \\ 
			\widetilde{z}(0) = z_0.
		\end{gather*}
		Moreover, assume that $\widehat{x} \in C([0,\tau]; X)$ and that there is a constant $R > 0$ with
		\[
			\|\widetilde{x}\|_{C([0,\tau]; X)} + \|\widehat{x}\|_{C([0,\tau]; X)} \leq R.
		\]
		Then, if $\widehat{x}$ satisfies $|||\widehat{x} - \widetilde{x}|||_{\tau} < \delta$ for $\delta > 0$ sufficiently small, there exists a unique solution $\widehat{z}$ on the time interval $[0, \tau]$ to
		\begin{gather*}
			\dot{\widehat{z}}(t) + \widetilde{\Gamma}(\widehat{z}(t), \widehat{x}(t)) + F(\widehat{z}(t)) = 0, \\ 
			\widehat{z}(0) = z_0
		\end{gather*}
		and there is a constant $C = C(R, |z_0|, \tau)$ such that
		\[
			\|\widetilde{z} - \widehat{z}\|_{C([0,\tau]; Z)} \leq C ||| \widetilde{x} - \widehat{x} |||_{\tau}^{1/2}.
		\]
	\end{lemma}
	
	\begin{proof}
		By classical local existence results, it is known that there is a maximal time $\widehat{\tau} \in [0,\tau]$ until which~$\widehat{z}$ exists, and that $\widehat{z}$ must blowup at time $\widehat{\tau}$ if $\widehat{\tau} \in (0,\tau)$. The strategy is to estimate for $t \in (0,\widehat{\tau})$ the difference $z(t) \coloneq \widetilde{z}(t) - \widehat{z}(t)$, which satisfies the equation
		\begin{equation}\label{equation:diffzzz}
			\dot{z} + \widetilde{\Gamma}(z, \widehat{x}) + \widetilde{\Gamma}(\widetilde{z}, \widetilde{x}-\widehat{x}) = F(\widehat{z}) - F(\widetilde{z})
		\end{equation}
		with initial condition $z(0) = 0$. First, we consider the third term in the left-hand side of \eqref{equation:diffzzz}, multiply by $z$, and observe that for all $t \in [0,\widehat{\tau})$ it holds
		\begin{equation*}
			\begin{aligned}
				\int_0^t \left[\widetilde{\Gamma}(\widetilde{z}, \widetilde{x}-\widehat{x}) z\right](s) \, {\rm d}s & = \int_0^t \widetilde{\Gamma}\left(\widetilde{z}(s), \frac{\rm d}{{\rm d}s} \int_0^s (\widetilde{x}(r)-\widehat{x}(r)) \, {{\rm d}r} \right) z(s) \, {{\rm d}s} \\
				& = \widetilde{\Gamma}\left(\widetilde{z}(t), \int_0^t (\widetilde{x}(r)-\widehat{x}(r)) \, {{\rm d}r} \right) z(t) \\
				& \quad - \int_0^t \widetilde{\Gamma}\left(\dot{\widetilde{z}}(s), \int_0^s (\widetilde{x}(r)-\widehat{x}(r)) \, {{\rm d}r} \right) z(s) \, {{\rm d}s} \\
				& \quad - \int_0^t \widetilde{\Gamma}\left(\widetilde{z}(s), \int_0^s (\widetilde{x}(r)-\widehat{x}(r)) \, {{\rm d}r} \right) \dot{z}(s) \, {{\rm d}s}.
			\end{aligned}
		\end{equation*}
		Moreover, for any $\widehat{T} \in (0, \widehat{\tau})$ there exists by assumption a compact subset of $Z$ containing $\{\widetilde{z}(q), \widehat{z}(r), \dot{\widetilde{z}}(s), \dot{\widehat{z}} \, | \, q,r,s,t \in [0,\widehat{T}]\}$. Thus, with a constant $C > 0$, depending on $R$, $z_0$, and $\tau$, it follows that
		\[
			\left| \int_0^t \widetilde{\Gamma}(\widetilde{z}, \widetilde{x}-\widehat{x}) z(s) \, {{\rm d}s} \right| \leq C  ||| \widetilde{x} - \widehat{x} |||_{\tau} \boxnote{\qquad\quad(t\in [0,\widehat{T}]).}
		\]
		As a result, denoting by $L_{\widehat{T}}  > 0$ the minimal Lipschitz constant for $F$ on the smallest ball containing the $1$-neighborhood of $\{\widetilde{z}(t), \widehat{z}(t) \, | \, t \in [0,\widehat{T}]\}$, depending on $\widehat{T}$, $z_0$, and $R$, it can be inferred for all $t \in [0,\widehat{T}]$ that
		\begin{equation*}
			\begin{aligned}
				\frac{|z(t)|^2}{2} & \leq \int_0^t L_{\widehat{T}} |z(s)|^2 \, {{\rm d}s} + \int_0^t |\widehat{x}(s)| |z(s)|^2 \, {{\rm d}s} \\
				& \quad + C ||| \widetilde{x} - \widehat{x} |||_{\tau},
			\end{aligned}
		\end{equation*}
		which yields together with Gr\"onwall's inequality that
		\[
			|z(t)|^2 \leq 2 C ||| \widetilde{x} - \widehat{x} |||_{\tau} \exp \left( 2\int_0^t \left(L_{\widehat{T}} + |\widehat{x}(s)|\right) \, {{\rm d}s} \right).
		\]
		The above estimate implies the announced conclusion, provided that we check that $\widehat{T}$ could be chosen as $\widehat{\tau}$ and that $\widehat{\tau} = \tau$ when $|||\widehat{x} - \widetilde{x}|||_{\tau} < \delta$ holds for $\delta > 0$ sufficiently small. To see this, we argue by contradiction and assume that for all $\delta_0 > 0$ there would exist $\delta \in (0,\delta_0)$ such that $\widehat{\tau} \in (0,\tau)$ despite $|||\widehat{x} - \widetilde{x}|||_{\tau} < \delta$. Then, as $\{\widetilde{z}(t) \, | \, t \in [0,\tau]\}$ is contained in a compact set, there would be a first time $\overline{\tau} \in (0, \tau)$ such that $|z(\overline{\tau})| = 1$. In this case, we fix the Lipschitz constant $L_{\overline{\tau}}$ in the same manner as above and note that
		\begin{equation}\label{equation:lsct}
			|z(\overline{\tau})|^2 \leq 2 C ||| \widetilde{x} - \widehat{x} |||_{\tau} \exp \left( 2\tau \left(L_{\overline{\tau}} + R\right)  \right),
		\end{equation}
		which would produce a contradiction by taking
		\[
			\delta_ 0 = \frac{1}{3C \exp \left( 2\tau \left(L_{\overline{\tau}} + R\right) \right)}.
		\]
		Here, we used that unlike $\overline{\tau}$ the constant $L_{\overline{\tau}}$ in the right-hand side of \eqref{equation:lsct} is actually independent of $\delta_0$, thanks to the definitions of $\overline{\tau}$ and $L_{\overline{\tau}}$. 
	\end{proof}

	\subsection*{Acknowledgments}
    MR is supported by a start-up grant from ShanghaiTech University. 

    \phantomsection	
	\addcontentsline{toc}{section}{References}

\bibliographystyle{alpha}
\bibliography{bib}

\end{document}